\documentclass[11pt,fleqn]{article}

\usepackage{amsmath,amssymb,amsthm,epsfig,graphics}
\usepackage[cp1251]{inputenc}
\usepackage{cite}

\def\C{\mathbb C}
\def\R{\mathbb R}
\def\Z{\mathbb Z}

\def\r{\rangle}
\def\l{\langle}

\def\a{\alpha}
\def\w{\omega}

\newcommand{\comment}[1]{}

\newtheorem*{lemma*}{Lemma}

\newtheorem{proposition}{Proposition}
\newtheorem*{conjecture*}{Conjecture}

{\theoremstyle{definition}
\newtheorem{definition}{Definition}
\newtheorem{example}{Example}
\newtheorem{remark}{Remark}

\newtheorem*{note*}{Note}
}

\allowdisplaybreaks

\begin{document}

\begin{flushleft}
\LARGE \bf Orbit functions of ${\rm SU}(n)$ and Chebyshev polynomials
\end{flushleft}

\begin{flushleft} \it
Maryna NESTERENKO~$^\dag$,
Ji\v{r}\'{i} PATERA~$^\ddag$
and
Agnieszka TERESZKIEWICZ~$^\S$
\end{flushleft}

\noindent $^\dag$~Institute of Mathematics of NAS of Ukraine, 3 Tereshchenkivs'ka Str., Kyiv-4, 01601 Ukraine

\noindent $^\ddag$~Centre de recherches math\'ematiques, Universit\'e de Montr\'eal, C.P.6128-Centre ville, Montr\'eal,\\
\phantom{$^\ddag$}~H3C\,3J7,~Qu\'ebec,~Canada

\noindent $^\S$~Institute of Mathematics, University of Bialystok, Akademicka 2, PL-15-267 Bialystok, Poland
\noindent $\phantom{^\S}$~E-mail: maryna@imath.kiev.ua, patera@crm.umontreal.ca, a.tereszkiewicz@uwb.edu.pl

{\vspace{9mm}\par\noindent\hspace*{8mm}\parbox{140mm} 
{\small 
Orbit functions of a simple Lie group/Lie algebra $L$ consist of
exponential functions summed up over the Weyl group of $L$. They are
labeled by the highest weights of irreducible finite dimensional
representations of $L$. They are of three types: $C$-, $S$- and
$E$-functions. Orbit functions of the Lie algebras $A_n$,
or equivalently, of the Lie group ${\rm SU}(n+1)$, are considered.

First, orbit functions in two different bases -- one orthonormal,
the other given by the simple roots of ${\rm SU}(n)$ -- are written using
the isomorphism of the permutation group of $n$ elements and
the Weyl group of ${\rm SU}(n)$.

Secondly, it is demonstrated that there is a one-to-one
correspondence between classical Chebyshev polynomials of the
first and second kind, and $C$- and $S$-functions of the
simple Lie group ${\rm SU}(2)$.

It is then shown that the well-known orbit functions of ${\rm SU}(n)$
are straightforward generalizations of Chebyshev polynomials to
$n-1$ variables. Properties of the orbit functions
provide a wealth of properties of the polynomials.

Finally, multivariate exponential functions are considered,
and their connection with orbit functions of ${\rm SU}(n)$ is established.}\par\vspace{7mm}}

\section{Introduction}\label{Introduction}

The history of the Chebyshev polynomials dates back over a century.
Their properties and applications have been considered in many papers.
We refer to \cite{Shahat,Rivlin1974} as a basic reference.
Studies of polynomials in more than one variable were undertaken
by several authors, namely \cite{Dunkl,DunnLidl1980,EierLidl,Koornwinder1-4,MasonHandscomb2003,suetin,suetin2}.
Of these, none follow the path we have laid down here.

In this paper, we demonstrate that the classical Chebyshev polynomials in one variable are naturally associated with the action of the Weyl group of ${\rm SU}(2)$, or equivalently with the action of the Weyl group $W(A_1)$ of the simple Lie algebra of type $A_1$. The association is so simple  that it has been ignored so far. However, by making
$W(A_1)$ the cornerstone of our rederivation of Chebyshev
polynomials, we have gained insight into the structure of the
theory of polynomials. In particular, the generalization of Chebyshev
polynomials to any number of variables was a straightforward task.
It is based on the Weyl group $W(A_n)$, where $n<\infty$.
This only recently became possible, after the orbit functions of simple
Lie algebras were introduced as useful special functions \cite{Patera}
and studied in great detail and generality \cite{KlimykPatera2006,KlimykPatera2007-1,KlimykPatera2008}.

We proceed in three steps. In Section~2, we exploit the
isomorphism of the group of permutations of~$n+1$ elements ${\rm S}$ and the
Weyl group of~${\rm SU}(n+1)$, or equivalently of~$A_n$, and define the
orbit functions of~$A_n$. This opens the possibility to write the
orbit functions in two rather different bases, the orthnormal basis,
and the basis determined by the simple roots of $A_n$, which
considerably alters the appearance of the orbit functions.
In the paper, we use the non-orthogonal basis because of its
direct generalization to simple Lie algebras of other types than~$A_n$.

In Section~3 we consider classical Chebyshev polynomials of
the first and second kind, and compare them with the $C$- and
$S$-orbit functions of~$A_1$. We show that polynomials of
the first kind are in one-to-one correspondence with
$C$-functions. Polynomials of the second kind coincide with the
appropriate $S$-function divided by the unique lowest non-trivial
$S$-function. We point out that polynomials of the second kind
can be identified as irreducible characters of finite dimensional
representations of ${\rm SU}(2)$. Useful properties of Chebyshev polynomials
can undoubtedly be traced to that identification, because the fundamental object
of representation theory of semisimple Lie groups/algebras is character.
In principle, all one needs to know about an irreducible finite
dimensional representation can be deduced from its character. An
important aspect of this conclusion is that characters are known and
uniformly described for all simple Lie groups/algebras.

In Section~4 we provide details of the recursive procedure from which the analog of the trigonometric form of Chebyshev polynomials in~$n$ variables can be found. Thus there are $n$ generic recursion relations for~$A_n$, having at least $n+2$ terms, and at most
$\left(\begin{smallmatrix}n+1\\ [(n+1)/2]\\\end{smallmatrix}\right)+1$ terms.  Irreducible polynomials are divided into $n+1$ exclusive classes with the property that monomials within one irreducible polynomial belong to the same class. This follows directly from the recognition of the presence and properties of the underlying Lie algebra.

In subsection 4.2, the simple substitution $z=e^{2\pi ix}$, $x\in\R^n$, is used in orbit functions to form analogs of  Chebyshev polynomials in~$n$ variables in their non-trigonometric form.  It is shown that, in the case of 2 variables, our polynomials coincide with those of Koorwider~\cite{Koornwinder1-4}(III), although the approach and
terminology could not be more different, ours being purely algebraic, having originated in Lie theory.

In Section~5, we present the orbit functions of~$A_n$ disguised as polynomials built from multivariate orbit functions of the symmetric group. In Section~2, such a possibility is described in terms of related bases, one orthonormal (symmetric group), the other
non-orthogonal (simple roots of $A_n$ and their dual $\omega$-basis). Both forms of the same polynomials appear rather different but may prove useful in different situations.

The last section contains a few comments and some questions related to the
subject of this paper that we find intriguing.

\section{Preliminaries}

This section is intended to fix notation and terminology. We also briefly recall some facts about ${\rm S}_{n+1}$ and $A_{n}$, dwelling particularly on various bases in $\R^{n+1}$ and
$\R^n$. In Section~\ref{ssec_Weyl_group}, we identify elementary reflections that generate the $A_n$ Weyl group $W$, with the permutation of two adjacent objects in an ordered set of $n+1$ objects. And, finally, we present some standard definitions and properties of orbit functions.

\subsection{Permutation group ${\rm S}_{n+1}$}

The group ${\rm S}_{n+1}$ of order $(n+1)!$ transforms the ordered number
set $[l_1,l_2,\dots,l_n,l_{n+1}]$ by permuting the numbers.

We introduce an orthonormal basis in the real Euclidean space $\R^{n+1}$,
\begin{equation}
{e_i}\in\R^{n+1}\,,\qquad
\l e_i , e_j\r=\delta_{ij}\,,\qquad    1\leq i,j\leq n+1\,,
\end{equation}
and use the $l_k$'s as the coordinates of a point $\mu$ in the $e$-basis:
$$
\mu=\sum_{k=1}^{n+1}l_ke_k\,,\qquad l_k\in\R\,.
$$

The group ${\rm S}_{n+1}$ permutes the coordinates $l_k$ of $\mu$, thus
generating other points from it. The set of all distinct points,
obtained by application of ${\rm S}_{n+1}$ to $\mu$, is called the orbit
of ${\rm S}_{n+1}$. We denote an orbit by $W_\lambda$, where $\lambda$ is
a unique point of the orbit, such that
$$
l_1\geq l_2\geq\cdots\geq l_n\geq l_{n+1}\,.
$$
If there is no pair of equal $l_k$'s in $\lambda$, the orbit $W_\lambda$ consists of
$(n+1)!$ points.

Further on, we will only consider points $\mu$ from the $n$-dimensional subspace
${\mathcal H}\subset\R^{n+1}$ defined by the equation
\begin{gather}\label{plane H}
\sum_{k=1}^{n+1}l_k=0.
\end{gather}

\subsection{Lie algebra~$A_n$}

Let us recall basic properties of the simple Lie algebra~$A_n$ of the compact
Lie group ${\rm SU}(n+1)$. Consider the general value $(1\leq n<\infty)$ of the rank.
The Coxeter-Dynkin diagram, Cartan matrix $\mathfrak{C}$, and inverse Cartan
matrix $\mathfrak{C}^{-1}$ of~$A_n$ are as follows:
\begin{gather*}
\parbox{.6\linewidth}
{\setlength{\unitlength}{1pt}
\def\kr{\circle{10}}
\thicklines
\begin{picture}(20,30)
\put(10,14){\kr}
\put(6,0){$\alpha_1$}
\put(15,14){\line(1,0){10}}
%
\put(30,14){\kr}
\put(26,0){$\alpha_2$}
\put(35,14){\line(1,0){10}}
%
\put(50,14){\kr}
\put(46,0){$\alpha_3$}
\put(55,14){\line(1,0){10}}
%
\put(70,13.5){$\ \,\ldots$}
%
\put(95,14){\line(1,0){10}}
%
\put(110,14){\kr}
\put(104,0){$\alpha_{n\!-\!1}$}
\put(115,14){\line(1,0){10}}
%
\put(130,14){\kr}
\put(126,0){$\alpha_n$}
\end{picture}
}
\hspace{-110 pt}
\mathfrak{C}{=}\left(\begin{smallmatrix}
\phantom{-}2&{-}1&\phantom{-}0&\phantom{-}0&\phantom{-}0&\dots&\phantom{-}0&\phantom{-}0&\phantom{-}0&\phantom{-}0\\
{-}1&\phantom{-}2&{-}1&\phantom{-}0&\phantom{-}0&\dots&\phantom{-}0&\phantom{-}0&\phantom{-}0&\phantom{-}0\\
\phantom{-}0&{-}1&\phantom{-}2&{-}1&\phantom{-}0&\dots&\phantom{-}0&\phantom{-}0&\phantom{-}0&\phantom{-}0\\[-1ex]
\vdots&\vdots&\vdots&\vdots&\vdots&\ddots&\vdots&\vdots&\vdots&\vdots\\[0.5ex]
\phantom{-}0&\phantom{-}0&\phantom{-}0&\phantom{-}0&\phantom{-}0&\dots&\phantom{-}0&{-}1&\phantom{-}2&{-}1\\
\phantom{-}0&\phantom{-}0&\phantom{-}0&\phantom{-}0&\phantom{-}0&\dots&\phantom{-}0&\phantom{-}0&{-}1&\phantom{-}2
\end{smallmatrix}\right),
\\[2ex]
\mathfrak{C}^{-1}=\frac{1}{n+1}\left(\begin{smallmatrix}
1\cdot n\;    &1\cdot(n-1)\;&1\cdot(n-2)\;&1\cdot(n-3)\;&\dots &1\cdot 3\;   &1\cdot 2\;   &1\cdot 1\\
1\cdot (n-1)\;&2\cdot(n-1)\;&2\cdot(n-2)\;&2\cdot(n-3)\;&\dots &2\cdot 3\;   &2\cdot 2\;   &2\cdot 1\\
1\cdot (n-2)\;&2\cdot(n-2)\;&3\cdot(n-2)\;&3\cdot(n-3)\;&\dots &3\cdot 3\;   &3\cdot 2\;   &3\cdot 1\\
1\cdot (n-3)\;&2\cdot(n-3)\;&3\cdot(n-3)\;&4\cdot(n-3)\;&\dots &4\cdot 3\;   &4\cdot 2\;   &4\cdot 1\\[-1 ex]
\vdots        &\vdots       &\vdots       &\vdots       &\ddots&\vdots       &\vdots       &\vdots\\[0.5ex]
1\cdot 3\;    &2\cdot3\;    &3\cdot3\;    &4\cdot3\;    &\dots &(n-2)\cdot3\;&(n-2)\cdot2\;&(n-2)\cdot 1\\
1\cdot 2\;    &2\cdot2\;    &3\cdot2\;    &4\cdot2\;    &\dots &(n-2)\cdot2\;&(n-1)\cdot2\;&(n-1)\cdot 1\\
1\cdot 1\;    &2\cdot1\;    &3\cdot1\;    &4\cdot1\;    &\dots &(n-2)\cdot1\;&(n-1)\cdot1\;&n\cdot 1\
                     \end{smallmatrix}\right).
\end{gather*}

The simple roots $\alpha_i$, $1\le i\le n$ of $A_n$ form a basis ($\a$-basis) of a
real Euclidean space $\R^n$. We choose them in ${\mathcal H}$:
\begin{gather*}
\alpha_i = e_i - e_{i+1}, \quad i=1,\dots,n.
\end{gather*}
This choice fixes the lengths and relative angles of the simple roots.
Their length is equal to~$\sqrt 2$ with relative angles between $\alpha_k$ and
$\alpha_{k+1}$  $(1\leq k\leq n-1)$ equal to~$\frac{2\pi}{3}$, and $\tfrac\pi2$
for any other pair.

In addition to $e$- and $\a$-bases, we introduce the $\w$-basis as the $\Z$-dual
basis to the simple roots~$\alpha_i$:
\begin{gather*}
\l\alpha_i,\w_j\r=\delta_{ij},\quad 1\leq i,j\leq n.
\end{gather*}
It is also a basis in the subspace ${\mathcal H}\subset\R^{n+1}$ (see~(\ref{plane H})).
The bases $\alpha$ and $\w$ are related by the Cartan matrix:
\begin{gather*}
\alpha=\mathfrak{C}\w,\quad \w=\mathfrak{C}^{-1}\alpha.
\end{gather*}

Throughout the paper, we use $\lambda\in\mathcal H$.
Here, we fix the notation for its coordinates relative to the $e$- and $\omega$-bases:
\begin{gather*}\label{notation_bases}
\lambda=\sum_{j=1}^{n+1}l_je_j=:(l_1,\ldots,l_{n+1})_e=\sum_{i=1}^n\lambda_i\w_i=:(\lambda_1,\ldots,\lambda_n)_\w,
\qquad  \sum_{i=1}^{n+1}l_i=0.
\end{gather*}

Consider a point $\lambda\in {\mathcal H}$ with coordinates $l_j$ and $\lambda_i$
in the $e$- and $\w$-bases, respectively. 
Using $\alpha=\mathfrak{C}\w$, i.e. $\w_i=\sum^n_{k=1}(\mathfrak{C}^{-1})_{ik}\alpha_k$,
we obtain the relations between $\lambda_i$ and $l_j$:
\begin{gather*}
l_1=\sum^n_{k=1}\lambda_k\mathfrak{C}^{{-}1}_{k1},\qquad
l_{n{+}1}=-\sum^n_{k=1}\lambda_k\mathfrak{C}^{{-}1}_{kn},
\\
l_j{=}\lambda_{1}(\mathfrak{C}^{{-}1}_{1\,j}{-}\mathfrak{C}^{{-}1}_{1\,j{-}1})
{+}\lambda_{2}(\mathfrak{C}^{{-}1}_{2\,j}{-}\mathfrak{C}^{{-}1}_{2\,j{-}1}){+}\!\cdots\!
{+}\lambda_n(\mathfrak{C}^{{-}1}_{n\,j}{-}\mathfrak{C}^{{-}1}_{n\,j{-}1}),
\quad j=2,\dots,n.
\end{gather*}
or explicitly,
\begin{gather}\label{l_to_lambda}
\lambda_i=l_i-l_{i+1},\qquad i=1,2,\ldots ,n.
\end{gather}
The inverse formulas are much more complicated
\begin{gather}\label{lambda_to_l}
l=A\lambda,
\end{gather}
where $l=(l_1,\dots,1_{n+1})$, $\lambda=(\lambda_1,\dots,\lambda_n)$, and $A$ is the $(n{+}1)\times n$ matrix:
\begin{gather*}
A=\tfrac{1}{n+1}\left(
\begin{smallmatrix}
n&n-1&n-2&\cdots&2&1\\
{-}1&n-1&n-2&\cdots&2&1\\
{-}1&{-}2&n-2&\cdots&2&1\\[-1 ex]
\vdots&\vdots&\vdots&\ddots&\vdots&\vdots\\
{-}1&{-}2&{-}3&\cdots&{-}(n-1)&1\\
{-}1&{-}2&{-}3&\cdots&{-}(n-1)&{-}n
\end{smallmatrix}
\right).
\end{gather*}

\subsection {The Weyl group of $A_n$}\label{ssec_Weyl_group}

The Weyl group $W(A_n)$ of order $(n+1)!$ acts in ${\mathcal H}$ by permuting
coordinates in the $e$-basis, i.e. as the group~${\rm S}_{n+1}$.
Indeed, let $r_i$, $1\le i\le n$ be the generating elements of $W(A_n)$, i.e,
reflections with respect to the hyperplanes perpendicular to $\alpha_i$ and
passing through the origin. Let $x=\sum\limits_{k=1}^{n+1}x_ke_k=(x_1,x_2,\dots,x_{n+1})_e$
and $\langle\cdot,\cdot\rangle$ denote the inner product.
We then have the~reflection by $r_i$:
\begin{equation}
\begin{aligned}
r_i x&= x-\tfrac{2}{\langle \alpha_i,\alpha_i\rangle}\langle x,\alpha_i\rangle\alpha_i
=(x_1,x_2,\dots,x_{n+1})_e-(x_i-x_{i+1})(e_i-e_{i+1})\\
&=(x_1,\dots,x_{i-1},x_{i+1},x_i,x_{i+2},\dots,x_{n+1})_e.
\end{aligned}
\end{equation}
Such transpositions generate the full permutation group ${\rm S}_{n+1}$.
Thus, $W(A_n)$ is isomorphic to ${\rm S}_{n+1}$, and the points
of the orbit $W_\lambda({\rm S}_{n+1})$ and $W_\lambda(A_n)$ coincide.

\subsection{Definitions of orbit functions}\label{ssec_def_orb_func}

The notion of an orbit function in $n$ variables depends essentially on
the underlying semisimple Lie group $G$ of rank $n$.
In our case, $G={\rm SU}(n+1)$ (equivalently, Lie algebra $A_n$).
Let the basis of the simple roots be denoted by
$\alpha$, and the basis of fundamental weights by~$\omega$.

The \textit{weight lattice} $P$ is formed by all integer linear combinations of the $\omega$-basis,
\begin{gather*}
P=\Z\omega_1+\Z\omega_2+\cdots+\Z\omega_n.
\end{gather*}
In the weight lattice $P$, we define the \textit{cone of dominant weights}
$P^+$ and its subset of strictly dominant weights $P^{++}$
\begin{gather*}
P\;\supset\; P^+=\Z^{\ge 0}\omega_1+\cdots+\Z^{\ge 0}\omega_n
\;\supset\; P^{++}=\Z^{>0}\omega_1+\cdots+\Z^{>0}\omega_n.
\end{gather*}

Hereafter, $W^e\subset W$ denotes the \textit{even subgroup} of the Weyl group
formed by an even number of reflections that generate $W$.
$W_{\lambda}$ and $W^e_{\lambda}$ are the corresponding group orbits of a point $\lambda\in \R^n$.

We also introduce the notion of fundamental region $F(G)\subset \R^n$.
For $A_n$ the \emph{fundamental region} $F$ is the convex hull of the vertices
$\{0,\omega_1,\omega_2,\ldots, \omega_n\}$.

\begin{definition}
The $C$ orbit function $C_{\lambda}(x)$, $\lambda\in P^+$ is defined as
\begin{gather}\label{def_c-function1}
C_\lambda(x) := \sum_{\mu\in W_\lambda(G)} e^{2\pi i \l\mu, x\r},
\qquad
x\in\R^n.
\end{gather}
\end{definition}

\begin{definition}
The $S$ orbit function $S_{\lambda}(x)$, $\lambda\in P^{++}$ is defined as
\begin{gather}\label{def_s-function1}
S_\lambda(x) := \sum_{\mu\in W_\lambda(G)} (-1)^{p(\mu)}e^{2\pi i \l\mu , x\r},\qquad
x\in\R^n,
\end{gather}
where $p(\mu)$ is the number of
reflections necessary to obtain $\mu$ from $\lambda$.
Of course the same $\mu$ can be obtained by different successions of reflections,
but all routes from $\lambda$ to $\mu$ will have a length of the same
parity, and thus the salient detail given by $p(\mu)$, in the
context of an $S$-function, is meaningful and unchanging.
\end{definition}

\begin{definition}
We define $E$ orbit function $E_{\lambda}(x)$, $\lambda\in P^e$ as
\begin{gather}\label{def_e-function1}
E_\lambda(x) := \sum_{\mu\in W^e_{\lambda}(G)} e^{2\pi i \l\mu, x\r},
\qquad
x\in\R^n,
\end{gather}
where $P^e:=P^+\cup r_i P^+$ and $r_i$ is a reflection from $W$.
\end{definition}

If we always suppose that $\lambda, \mu\in P$ are given in the $\w$-basis,
and $x\in\R^n$ is given in the $\alpha$ basis, namely
$\lambda=\sum\limits^n_{j=1}\lambda_j\w_j$,
$\mu=\sum\limits^n_{j=1}\mu_j\w_j$, $\lambda_j, \mu_j\in\Z$ and
$x=\sum\limits^n_{j=1}x_j\alpha_j$, $x_j\in \R$,
then the orbit functions of~$A_n$ have the following forms
\begin{gather}
C_\lambda(x)
= \sum_{\mu\in W_\lambda} e^{2\pi i \sum\limits^n_{j=1}\mu_jx_j}
= \sum_{\mu\in W_\lambda} \prod\limits^n_{j=1} e^{2\pi i \mu_jx_j},\label{def_c-function2}
\\
S_\lambda(x)
= \sum_{\mu\in W_\lambda} (-1)^{p(\mu)}e^{2\pi i \sum\limits^n_{j=1}\mu_jx_j}
= \sum_{\mu\in W_\lambda} (-1)^{p(\mu)}\prod\limits^n_{j=1} e^{2\pi i \mu_jx_j},\label{def_s-function2}
\\
E_\lambda(x)
= \sum_{\mu\in W^e_{\lambda}} e^{2\pi i \sum\limits^n_{j=1}\mu_jx_j}
= \sum_{\mu\in W^e_{\lambda}} \prod\limits^n_{j=1} e^{2\pi i \mu_jx_j}.\label{def_e-function2}
\end{gather}

\subsection{Some properties of orbit functions}

For $S$ functions, the number of summands is always equal to the size
of the Weyl group. Note that in the 1-dimensional case, $C$-, $S$-
and $E$-functions are respectively a cosine, a sine and an
exponential functions up to the constant.

All three families of orbit functions are based on semisimple Lie algebras.
The number of variables coincides with the rank of the Lie algebra.
In general, $C$-, $S$- and $E$- functions are finite sums of
exponential functions. Therefore they are continuous and have
continuous derivatives of all orders in $\R^n$.

The $S$-functions are antisymmetric with respect to the
$(n{-}1)$-dimensional boundary of~$F$. Hence they are zero on the
boundary of~$F$. The $C$-functions are symmetric with respect to the
$(n-1)$-dimensional boundary of~$F$. Their normal derivative at the
boundary is equal to zero (because the normal derivative of a $C$-function
is an $S$-function).

For simple Lie algebras of any type,
the functions $C_\lambda(x)$, $E_\lambda(x)$ and $S_\lambda(x)$ are eigenfunctions
of the appropriate Laplace operator. The Laplace operator has the same eigenvalues
on every exponential function summand of an orbit function with
eigenvalue~$-4\pi\langle \lambda,\lambda\rangle$.

\subsubsection{Orthogonality}

For any two complex squared integrable functions $\phi(x)$ and $\psi(x)$
defined on the fundamental region $F$, we define a continuous scalar product
\begin{gather}\label{def_cont_scalar_product}
\l\phi(x),\psi(x)\r:=\int\limits_{{F}}\phi(x)\overline{\psi(x)}{\rm d}x.
\end{gather}
Here, integration is carried out with respect to the Euclidean measure,
the bar means complex conjugation and $x\in {F}$, where ${F}$ is the fundamental
region of either $W$ or $W^e$ (note that the fundamental region of $W^e$ is
$F^e=F\cup r_i F$, where $r_i\in W$).

Any pair of orbit functions from the same family is orthogonal on the corresponding
fundamental region with respect to the scalar product~(\ref{def_cont_scalar_product}), namely
\begin{gather}\label{cont_orthog c funcs}
\l C_{\lambda}(x),C_{\lambda'}(x)\r=|W_\lambda|\cdot|F|\cdot\delta_{\lambda\lambda'},
\\\label{cont_orthog s funcs}
\l S_{\lambda}(x),S_{\lambda'}(x)\r=|W|\cdot|F|\cdot\delta_{\lambda\lambda'},
\\\label{cont_orthog e funcs}
\l E_{\lambda}(x),E_{\lambda'}(x)\r=|W^e_{\lambda}|\cdot|F^e|\cdot\delta_{\lambda\lambda'},
\end{gather}
where $\delta_{\lambda\lambda'}$ is the Kronecker delta, $|W|$ is the order of the Weyl group,
$|W_{\lambda}|$ and $|W^e_{\lambda}|$ are the sizes of the Weyl group orbits (the number of distinct  points in the orbit), and $|F|$ and $|F^e|$ are volumes of fundamental regions. The volume $|F|$  was calculated in~\cite{HrivnakPatera2009}.
\begin{proof}
Proof of the relations~(\ref{cont_orthog c funcs},\ref{cont_orthog s funcs},\ref{cont_orthog e funcs})  follows from the orthogonality of the usual exponential functions and from the fact that a given weight $\mu\in P$ belongs to precisely one orbit function.
\end{proof}

The families of $C$-, $S$- and $E$-functions are complete on the fundamental domain.
The completeness of these systems follows from the completeness of the system of
exponential functions; i.e., there does not exist a function $\phi(x)$, such that
$\langle\phi(x),\phi(x)\rangle>0$, and at the same time $\langle\phi(x),\psi(x)\rangle=0$
for all functions $\psi(x)$ from the same system.

\subsubsection{Orbit functions of $A_n$ acting in $\R^{n+1}$}\label{orb_func_n+1}

Relations~(\ref{lambda_to_l}) allow us to rewrite variables $\lambda$ and $x$ in an orbit function
in the $e$-basis. Therefore we can obtain the $C$-, $S$- and $E$- functions acting in $\R^{n+1}$
\begin{gather}
C_\lambda(x) = \sum_{s\in {\rm S}_{n+1}} e^{2\pi  i  (s(\lambda), x)},\\
C_\lambda(x) = \sum_{s\in {\rm S}_{n+1}} ({\rm sgn}\,s)e^{2\pi  i  (s(\lambda), x)},\\
E_\lambda(x) = \sum_{s\in  {\rm Alt}_{n{+}1}}e^{2\pi  i (s(\lambda), x)},
\end{gather}
where $(\cdot\,,\cdot)$ is a scalar product in $\R^{n+1}$, ${\rm sgn}\,s$ is the permutation sign,
and ${\rm Alt}_{n{+}1}$ is the alternating group acting on an $(n+1)$-tuple of numbers.
Note that variables $x$ and $\lambda$ are in the hyperplane~$\mathcal{H}$.

Using the identity $\l\lambda,r_ix\r=\l r_i\lambda,x\r$ for the reflection $r_i$, $i=1,\ldots,n$,
it can be verified that
\begin{gather}\label{ri_symm_of_C_S}
C_\lambda(r_ix)=C_{r_i\lambda}(x)=C_\lambda(x),\quad
\text{and} \quad
{S}_{r_i\lambda}(x)={S}_\lambda(r_ix)=-{S}_\lambda(x).
\end{gather}

Note that it is easy to see for generic points that ${E}_\lambda(x)=\tfrac12\Big({C}_\lambda(x)+{S}_\lambda(x)\Big)$,
and from the relations~(\ref{ri_symm_of_C_S}), we obtain
\begin{gather}\label{ri_symm_of_E}
{E}_{r_i\lambda}(x)={E}_\lambda(r_ix)=\tfrac12\left( {C}_\lambda(x)-{S}_\lambda(x)\right)={E}_{\lambda}(x).
\end{gather}

A number of other properties of orbit functions are presented in~\cite{KlimykPatera2006, KlimykPatera2007-1, KlimykPatera2008}.

\section{Orbit functions and Chebyshev polynomials}\label{sec_A1_and_Chebyshev}

We recall known properties of Chebyshev polynomials \cite{Rivlin1974} in order to be subsequently able
to make an unambiguous comparison between them and the appropriate orbit functions.

\subsection{Classical Chebyshev polynomials}\label{ssec_Classical_Chebyshev}

Chebyshev polynomials are orthogonal polynomials which are usually defined recursively.
One distinguishes between Chebyshev polynomials of the first kind $T_n$:
\begin{gather}
T_0(x)=1,\quad
T_1(x)=x,\quad
T_{n+1}(x)=2xT_n-T_{n-1},\label{def_T0_T1_Tn}\\
\text{hence}\quad
T_2(x)=2x^2-1,\quad
T_3(x)=4x^3-3x, \dots\label{def_T2_T3}
\end{gather}
and Chebyshev polynomials of the second kind $U_n$:
\begin{gather}
U_0(x)=1,\quad
U_1(x)=2x,\quad
U_{n+1}(x)=2xU_n-U_{n-1},\label{def_U0_U1_Un}\\
\text{in particular}\quad
U_2(x)=4x^2-1,\quad
U_3(x)=8x^3-4x, \quad etc.\label{def_U2_U3}
\end{gather}

The polynomials $T_n$ and $U_n$ are of degree $n$ in the variable $x$.
All terms in a polynomial have the parity of $n$.
The coefficient of the leading term of $T_n$ is $2^{n-1}$ and $2^n$ for $U_n$, $n=1,2,3,\dots$.

The roots of the Chebyshev polynomials of the first kind are widely used as nodes
for polynomial interpolation in approximation theory. The Chebyshev polynomials are
a special case of Jacobi polynomials. They are orthogonal with the following weight functions:
\begin{gather}
\int\limits^1_{-1}\frac{1}{\sqrt{1-x^2}}T_n(x)T_m(x){\rm d}x=
\left\{
\begin{array}{l}
0,\quad n\ne m,\\
\pi,\quad n=m=0,\\
\frac{\pi}{2},\quad n=m\ne 0,
\end{array}
\right.
\\
\int\limits^1_{-1}\sqrt{1-x^2}U_n(x)U_m(x){\rm d}x=
\left\{
\begin{array}{l}
0,\quad n\ne m,\\
\frac{\pi}{2},\quad n=m.
\end{array}
\right.
\end{gather}

There are other useful relations between Chebyshev polynomials of the first and second kind.
\begin{gather}
\frac{\rm d}{{\rm d}x} T_n(x)=n U_{n-1}(x),\quad n=1,2,3,\dots \label{differentiation_Tn}
\\
T_n(x)=\frac 12 (U_{n}(x)-U_{n-2}(x)),\quad n=2,3,\dots \label{difference_Un_Un-2}
\\
T_{n+1}(x)=xT_n(x)-(1-x^2)U_{n-1},\quad n=1,2,3,\dots \label{difference_xTn_Un-1}
\\
T_{n}(x)=U_n(x)-xU_{n-1},\quad n=1,2,3,\dots \label{difference_Un_xUn-1}
\end{gather}

\subsubsection{Trigonometric form of Chebyshev polynomials}

Using trigonometric variable $x=\cos y$, polynomials of the first kind become
\begin{gather}\label{trigon_T}
T_n(x)=T_n(\cos y)=\cos(ny), \quad n=0,1,2,\dots
\end{gather}
and polynomials of the second kind are written as
\begin{gather}\label{trigon_U}
U_n(x)=U_n(\cos y)=\frac{\sin((n+1)y)}{\sin y}, \quad n=0,1,2,\dots
\end{gather}

For example, the first few lowest polynomials are
\begin{gather*}
T_0(x)=T_0(\cos y)=\cos(0y)=1, \quad
T_1(x)=T_1(\cos y)=\cos(y)=x, \\
T_2(x)=T_2(\cos y)=\cos(2y)=\cos^2y-\sin^2y=2cos^2y-1=2x^2-1; \quad \\
U_0(x)=U_0(\cos y)=\frac{\sin y}{\sin y}=1, \quad
U_1(x)=U_1(\cos y)=\frac{\sin(2y)}{\sin y}=2\cos y=2x, \\
U_2(x)=U_2(\cos y)=\frac{\sin(3y)}{\sin y}=\frac{\sin(2y)\cos y+\sin y \cos(2y)}{\sin y}
=4\cos^2y-1=4x^2-1.
\end{gather*}

\subsection{Orbit functions of $A_1$ and Chebyshev polynomials}

Let us consider the orbit functions of one variable. There is only one simple Lie algebra of
rank~1, namely $A_1$. Our aim is to build the recursion relations in a way that generalizes
to higher rank groups, unlike the standard relations of the classical theory presented above.

\subsubsection{Orbit functions of $A_1$ and trigonometric form of $T_n$ and $U_n$}

The orbit of $\lambda=m\omega_1$ has two points for $m\ne 0$, namely $W_\lambda=\{(m),(-m)\}$.
The orbit of $\lambda=0$ has just one point, $W_0=\{0\}$.

One-dimensional orbit functions have the form
(see~(\ref{def_c-function2}), (\ref{def_s-function2}), (\ref{def_e-function2}))
\begin{gather}
C_\lambda(x)=e^{2\pi i m x}{+}e^{{-}2\pi i m x}{=}2\cos(2\pi m x){=}2\cos(m y),
\quad \text{where}\quad y=2\pi x,\; m\in\Z^{\geqslant 0};\label{c-func_A1}
\\
S_\lambda(x)=e^{2\pi i m x}{-}e^{-2\pi i m x}{=}2i\sin(2\pi m x){=}2i\sin(m y),
\quad \text{for}\quad m\in\Z^{> 0};\label{s-func_A1}
\\
E_\lambda(x)= e^{2\pi i m x}=y^m,
\quad \text{where}\quad y=e^{2\pi i x}, \; m\in\Z.\label{e-func_A1}
\end{gather}
From~(\ref{c-func_A1}) and~(\ref{trigon_T}) it directly follows that polynomials generated from
$C_m$ functions of $A_1$ are doubled Chebyshev polynomials $T_m$ of the first kind for $m=0,1,2,\dots$.

Analogously, from~(\ref{s-func_A1}) and~(\ref{trigon_U}), it follows that
polynomials $\frac{S_{m+1}}{S_1}$ are Chebyshev polynomials $U_m$ of the second kind for $m=0,1,2,\dots$.

The polynomials generated from $E_m$ functions of $A_1$, form a standard monomial sequence
$y^m$, $m=0,1,2\dots$, which is the basis for the vector space of polynomials.

$C$- and $S$-orbit functions are orthogonal on the interval $F=[0,1]$
(see~(\ref{cont_orthog c funcs}) and~(\ref{cont_orthog s funcs}))
what implies the orthogonality of the corresponding polynomials.

Comparing the properties of one-dimensional orbit functions with properties of Chebyshev polynomials, we conclude that there is a one-to-one correspondence between the Chebyshev polynomials and the orbit functions.

\subsubsection{Orbit functions of $A_1$ and their polynomial form}\label{recursive_polynomials}

In this subsection, we start a derivation of the $A_1$ polynomials in a way which emphasizes the role of the Lie algebra and, more importantly, in a way that directly generalizes to simple Lie algebras  of any rank $n$ and any type, resulting in polynomials of $n$ variables and of a new type for each algebra. In the present case of $A_1$, this leads us to a different normalization of the polynomials and their trigonometric variables than is common for classical Chebyshev polynomials. No new polynomials emerge than those equivalent to Chebyshev polynomials of the first and second kind. Insight is nevertheless gained into the structure of the problem, which, to us, turned out to be of considerable importance. We are inclined to consider the Chebyshev polynomials, in the form derived here, as the canonical polynomials.

The underlying Lie algebra $A_1$ is often denoted $sl(2,\C)$ or $su(2)$. In fact, this case is so simple that the presence of the Lie algebras has never been acknowledged.

The orbit functions of $A_1$ are of two types~(\ref{c-func_A1}) and (\ref{s-func_A1});
in particular, $C_0(x)=2$, and $S_0(x)=0$ for all $x$.

The simplest substitution of variables to transform the orbit functions into polynomials is $y=e^{2\pi i x}$,  monomials in such a polynomial are $y^m$ and $y^{-m}$. Instead, we introduce new (`trigonometric') variables $X$ and $Y$ as follows:
\begin{gather}
X:=C_1(x){=}e^{2\pi i x}{+}e^{{-}2\pi i x}{=}2\cos(2\pi x),\label{def_X}\\
Y:=S_1(x){=}e^{2\pi i x}{-}e^{{-}2\pi i x}{=}2i\sin(2\pi x).\label{def_Y}
\end{gather}
We can now start to construct polynomials recursively in the degrees of $X$ and $Y$, by calculating the products of the appropriate orbit functions. Omitting the dependence on $x$ from the symbols, we have
\begin{gather}
\begin{array}{rlcrl}
X^2 \!\!\!&=C_2+2             &\Longrightarrow\quad&     C_2\!\!\!&=X^2-2,\\
XC_2\!\!\!&=C_3+X             &\Longrightarrow\quad&     C_3\!\!\!&=X^3-3X,\\
XC_m\!\!\!&=C_{m+1}+C_{m-1}   &\Longrightarrow\quad& C_{m+1}\!\!\!&=XC_m-C_{m-1},\quad m\geq3.
\end{array}\label{C_via_X}
\end{gather}
Therefore, we obtain the following recursive polynomial form of the $C$-functions
\begin{gather}\label{C-form_of_Tn}
C_0=2,\quad
C_1=X,\quad
C_2=X^2{-}2,\quad
C_3=X^3{-}3X,\quad
C_4=X^4{-}4X^2+2,\dots .
\end{gather}
After the substitution $z=\tfrac 12 X$ we have
\begin{gather*}
C_0{=}2\cdot1,\quad
C_1{=}2z,\quad
C_2{=}2(2z^2{-}1),\quad
C_3{=}2(4z^3{-}3z),\quad
C_4{=}2(8z^4{-}8x^2+1),\dots.
\end{gather*}
Hence we conclude that $C_m\!=2T_m$, for $m=0,1,\dots$.
\begin{remark}\

In our opinion, the normalization of orbit functions is also more `natural' for the Chebyshev polynomials.  For example, the equality $C_2^2=C_4+2$ does not hold for $T_2$ and $T_4$.
\end{remark}
\begin{remark}\

Each $C_m$ also can be written as a polynomial of degree $m$ in~$X$,$Y$ and~$S_{m-1}$.
It suffices to consider the products $YS_m$, e.g., $C_2=Y^2+2$, $C_3=YS_2+X$, etc.
Equating the polynomials obtained in such a way with the corresponding polynomials from~(\ref{C_via_X}), we obtain a trigonometric identity for each $m$. For example, we find two ways to write $C_2$, one from the product $X^2$ and one from $Y^2$. Equating the two, we get
\begin{gather*}
X^2{-}Y^2=4\quad\Longleftrightarrow\quad \sin^2(2\pi x){+}\cos^2(2\pi x)=1
\end{gather*}
because $Y$ is defined in~(\ref{def_Y}) to be purely imaginary.
\end{remark}

Just as the polynomials representing $C_m$ were obtained above,
it is possible to to find polynomial expressions for $S_m$ for all $m$.

Fundamental relations between the $S$- and $C$- orbit functions follow from the properties
of the character $\chi_m(x)$ of the irreducible representation of $A_1$ of dimension $m+1$.

The character can be written in two ways: as in the Weyl character formula and also as the sum
of appropriate $C$-functions. Explicitly, we have the $A_1$ character:
\begin{gather*}
\chi_m(x)=\frac{S_{m+1}(x)}{S_1(x)}=C_m(x)+C_{m-2}(x)+\cdots+
    \begin{cases}
    C_2(x)+1\quad    &\text {for $m$ even},\\
    C_3(x)+C_1(x)\quad &\text {for $m$ odd}.
    \end{cases}
\end{gather*}
Let us write down a few characters
\begin{gather*}
\chi_0=\tfrac{S_{1}(x)}{S_1(x)}=1,\quad
\chi_1=\tfrac{S_{2}(x)}{S_1(x)}=C_1=X,\quad
\chi_2=\tfrac{S_{3}(x)}{S_1(x)}=C_2+C_0=X^2-1,\\
\chi_3=\tfrac{S_{4}(x)}{S_1(x)}=C_3+C_1=X^3-2X,\quad
\chi_4=\tfrac{S_{5}(x)}{S_1(x)}=C_4+C_2+C_0=X^4-3X^2+1,\dots
\end{gather*}
Again, the substitution $z=\tfrac 12 X$ transforms these polynomials into the Chebyshev polynomials of the second kind $\frac{S_{m+1}}{S_1}=U_m$, \ $m=0,1,\dots$, indeed
\begin{gather*}
\tfrac{S_{1}(x)}{S_1(x)}=1,\quad
\tfrac{S_{2}(x)}{S_1(x)}=2z,\quad
\tfrac{S_{3}(x)}{S_1(x)}=4z^2{-}1,\quad
\tfrac{S_{4}(x)}{S_1(x)}=8z^3{-}4z,\quad
\tfrac{S_{5}(x)}{S_1(x)}=16z^4{-}12z^2{+}1,\dots
\end{gather*}

\begin{remark}\

Note that in the character formula we used $C_0=1$, while above (see~(\ref{def_e-function2}) and
(\ref{C-form_of_Tn})) we used $C_0=2$. It is just a question of normalization of orbit functions.
For some applications/calculations it is convenient to scale orbit functions of non-generic
points on the factor equal to the order of the stabilizer of that point in the Weyl group $W(A_1)$.
\end{remark}

\section{Orbit functions of $A_n$ and their polynomials}

This section proposes two approaches to constructing orthogonal polynomials of $n$
variables based on orbit functions. The first comes from the decomposition of Weyl
orbit products into sums of orbits. Its result is the analog of the trigonometric form
of the Chebyshev polynomials. The second approach is the exponential substitution
in~\cite{KlimykPatera2006}.

\subsection{Recursive construction}

Since the $C$- and $S$- functions are defined for $A_n$ of any rank $n=1,2,3,\dots$, it is natural to take $C$-functions and the ratio of $S$-functions as multidimensional generalizations of Chebyshev polynomials of the first and second kinds respectively
\begin{align*}
T_\lambda(x)&:=C_\lambda(x),\qquad \ x\in\R^n,\\
U_\lambda(x)&:=\tfrac{S_{\lambda+\rho}(x)}{S_{\rho}(x)},\qquad
\rho=\omega_1{+}\omega_2{+}\dots{+}\omega_n=(1,1,\dots,1)_\omega,\quad x\in\R^n,
\end{align*}
where $\lambda$ is one of the dominant weights of $A_n$.

The functions $T_\lambda$ and $U_\lambda$ can be constructed as
polynomials using the recursive scheme proposed in
Section~\ref{recursive_polynomials}. In the $n$-dimensional case of
orbit functions of $A_n$, we start from the $n$ orbit functions
labeled by the fundamental weights,
\begin{gather*}
X_1:=C_{\omega_1}(x),\quad X_2:=C_{\omega_2}(x),\quad\dots,\quad
X_n:=C_{\omega_n}(x)\,,
        \qquad x\in\R^n\,.
\end{gather*}
By multiplying them and decomposing the products into the sum of orbit
functions, we build the polynomials for any $C$- and $S$-function.

The generic recursion relations are found as the decomposition of the products $X_{\w_j}C_{(a_1,a_2.\dots,a_n)}$ with `sufficiently large' $a_1,a_2,\dots,a_n$. Such a recursion relation has  $\left(\begin{smallmatrix}n+1\\ j\\\end{smallmatrix}\right)+1$ terms, where
$\left(\begin{smallmatrix}n+1\\j\\\end{smallmatrix}\right)$ is the size of the orbit of $\w_j$.

An efficient way to find the decompositions is to work with products
of Weyl group orbits, rather than with orbit functions. Their
decomposition has been studied, and many examples have been described in \cite{HLP}.
It is useful to be aware of the congruence class of
each product, because all of the orbits in its decomposition necessarily
belong to that class. The \emph{congruence number} $\#$ of an orbit
$\lambda$ of $A_n$, which is also the congruence number of the orbit
functions $C_\lambda$ and $S_\lambda$, specifies the class. It is
calculated as follows,
\begin{gather}
\#(C_{(a_1,a_2,\dots,a_n)}(x))=\#(S_{(a_1,a_2,\dots,a_n)}(x))=\sum_{k=1}^n ka_k\mod (n+1).
\end{gather}
In particular, each $X_j$, where $j=1,2,\dots,n$, is in its own
congruence class. During the multiplication, congruence numbers add
up $\mod n+1$.

Polynomials in two and three variables originating from orbit functions of the simple
Lie algebras $A_2$, $C_2$, $G_2$, $A_3$, $B_3$, and $C_3$ are obtained in the forthcoming
paper~\cite{NesterenkoPatera2009}.

\subsection{Exponential substitution}

There is another approach to multivariate orthogonal polynomials, which is also based on orbit functions. Such polynomials can be constructed by the continuous and invertible change of variables
\begin{gather}\label{subst_exp}
y_j=e^{2\pi i x_j}, \quad x_j\in\R,\quad j=1,2,\dots,n.
\end{gather}
Consider an $A_n$ orbit function $C_\lambda(x)$, $S_\lambda(x)$ or $E_\lambda(x)$, when $\lambda$ is
given in the $\w$-basis and $x$ is given in the $\alpha$-basis. Each of these functions consists
of summands $\prod\limits^n_{j=1}e^{2\pi i \mu_j x_j}$, where $\mu_j\in\Z$ are coordinates of an
orbit point $\mu$. Then the summand is transformed by (\ref{subst_exp}) into a monomial of the form
$\prod\limits^n_{j=1}y_j^{\mu_j}$. It is convenient to label these polynomials by non-negative
integer coordinates $(m_1,m_2,\dots, m_n)$ of the point $\lambda=m_1\w_1+m_2\w_2+\dots m_n\w_n$ and to
denote the polynomial obtained from the orbit function $C_\lambda$ as $P_{(m_1,\dots, m_n)}^{C}$
(analogously for $S$ and $E$ functions). Polynomials of two variables obtained from the orbit functions
by the substitution~(\ref{subst_exp}) are already described in the literature \cite{Koornwinder1-4}, where
they are derived from very different considerations. The detailed comparison is made in the following example.

\begin{example}
Consider the $A_2$ Weyl orbits of the lower weights $(0,m)_\omega$, $(m,0)_\omega$ and the
orbit of the generic point $(m_1,m_2)_\omega$, $m,m_1,m_2\in\Z^{>0}$
\begin{gather*}
W_{(0,m)}(A_2)=\{(0, m),\, ( {-}m, 0),\, (m, {-}m)\},
\quad W_{(m,0)}(A_2)=\{(m, 0),\, ({-}m, m),\, (0, {-}m)\},
\\
W_{(m_1,m_2)}(A_2)=\{
(m_1, m_2)^+,\ ({-}m_1, m_1{+}m_2)^-,\ (m_1{+}m_2, {-}m_2)^-,
\\
\phantom{W_{(m_1,m_2)}(A_2)=\{}
({-}m_2, {-}m_1)^-,\ ( {-}m_1{-}m_2, m_1)^+,\ (m_2, {-}m_1{-}m_2)^+\}.
\end{gather*}
Suppose $x=(x_1,x_2)$ is given in the $\alpha$-basis, then the orbit functions assume the form
\begin{gather}
\begin{gathered}\label{C_S_for_A2}
C_{(0,0)}(x)=1,\quad
C_{(0,m)}(x)=\overline{C_{(m,0)}(x)}=e^{{-}2\pi i mx_1}{+}e^{2\pi i mx_1}e^{{-}2\pi i mx_2}{+}e^{2\pi i mx_2},
\\
C_{(m_1,m_2)}(x)
=e^{2\pi i m_1x_1}e^{2\pi i m_2x_2}{+}e^{{-}2\pi i m_1x_1}e^{2\pi i (m_1{+}m_2)x_2}
{+}e^{2\pi i (m_1{+}m_2)x_1}e^{{-}2\pi i m_2x_2}{+}
\\ \phantom{C_{(m_1,m_2)}(x)=}
e^{{-}2\pi i m_2x_1}e^{{-}2\pi i m_1x_2}{+}e^{{-}2\pi i (m_1{+}m_2)x_1}e^{2\pi i m_1x_2}
{+}e^{2\pi i m_2x_1}e^{{-}2\pi i (m_1{+}m_2)x_2},
\\
S_{(m_1,m_2)}(x)
=e^{2\pi i m_1x_1}e^{2\pi i m_2x_2}{-}e^{{-}2\pi i m_1x_1}e^{2\pi i (m_1{+}m_2)x_2}
{-}e^{2\pi i (m_1{+}m_2)x_1}e^{{-}2\pi i m_2x_2}{-}
\\ \phantom{S_{(m_1,m_2)}(x)=}
e^{{-}2\pi i m_2x_1}e^{{-}2\pi i m_1x_2}{+}e^{{-}2\pi i (m_1{+}m_2)x_1}e^{2\pi i m_1x_2}
{+}e^{2\pi i m_2x_1}e^{{-}2\pi i (m_1{+}m_2)x_2}.
\end{gathered}
\end{gather}
Using~(\ref{subst_exp}) we have the following corresponding polynomials
\begin{gather*}
P_{(0,0)}^C=1,\qquad
P_{0,m}^C=\overline{P_{0,m}^C}=y_1^{{-}m}{+}Y_1^{m}y_2^{{-}m}{+}y_2^{m},
\\
\begin{gathered}\label{Polyn_c-A2}
P_{(m_1,m_2)}^{C}
=y_1^{m_1}y_2^{m_2}{+}y_1^{-m_1}y_2^{(m_1{+}m_2)}{+}y_1^{(m_1{+}m_2)}y_2^{-m_2}{+}
\\ \phantom{P^{C(A_2)}=}
y_1^{-m_1}y_2^{-m_2}{+}y_1^{-(m_1{+}m_2)}y_2^{m_1}{+}y_1^{m_2}y_2^{-(m_1{+}m_2)},
\end{gathered}
\\
\begin{gathered}\label{Polyn_s-A2}
P_{(m_1,m_2)}^{S}
=y_1^{m_1}y_2^{m_2}{-}y_1^{-m_1}y_2^{(m_1{+}m_2)}{-}y_1^{(m_1{+}m_2)}y_2^{-m_2}{-}
\\ \phantom{P^{S(A_2)})=}
y_1^{-m_1}y_2^{-m_2}{+}y_1^{-(m_1{+}m_2)}y_2^{m_1}{+}y_1^{m_2}y_2^{-(m_1{+}m_2)}.
\end{gathered}
\end{gather*}

The polynomials $e^+$ and $e^-$ given in (2.6)
of~\cite{Koornwinder1-4}(III) coincide with those
in~(\ref{C_S_for_A2}) whenever the correspondence $\sigma=2\pi x_1$,
$\tau=2\pi x_2$ is set up. So, both the orbit functions polynomials
of $A_2$ and $e^{\pm}$ are orthogonal on the interior of
Steiner's hypocycloid.

It is noteworthy that the regular tessellation of the plane by
equilateral triangles considered in~\cite{Koornwinder1-4} is the
standard tiling of the weight lattice of $A_2$. The fundamental
region $R$ of ~\cite{Koornwinder1-4} coincides with the fundamental
region $F(A_2)$ in our notations. The corresponding isometry group
is the affine Weyl group of~$A_2$.

Furthermore, continuing the comparison with the
paper~\cite{Koornwinder1-4}, we want to point out that orbit
functions are eigenfunctions not only of the Laplace operator
written in the appropriate basis, e.g. in $\omega$-basis, the
corresponding eigenvalues bring $-4\pi^2 \l\lambda ,\lambda\r$, where
$\lambda$ is the representative from the dominant Weyl chamber,
which labels the orbit function. This property holds not only for
the Lie algebra $A_n$ and its Laplace operator, but also for the
differential operators built from the elementary symmetric
polynomials, see~\cite{KlimykPatera2006, KlimykPatera2007-1}.
\end{example}

An independent approach to the polynomials in two variables is
proposed in~\cite{suetin2}, and the generalization of classical Chebyshev
polynomials to the case of several variables is also presented
in~\cite{EierLidl}. A detailed comparison would be a major task
because the results are not explicit and contain no examples of
polynomials.

\section{Multivariate exponential functions}

In this section, we consider one more class of special functions, which,
as it will be shown, are closely related to orbit functions of $A_n$.
Such a relation allows us to view orbit functions in the orthonormal basis, and
to represent them in the form of determinants and permanents.
At the same time, we obtain the straightforward procedure for constructing
polynomials from multivariate exponential functions.

\begin{definition}\cite{KlimykPatera2007-3}
For a fixed point $\lambda=(l_1,l_2,\dots,l_{n+1})_e$, such that
$l_1\geq l_2\geq\cdots\geq l_{n+1}$, $\sum\limits_{k=1}^{n+1}l_k=0$,
the symmetric multivariate exponential function $D^+_{\lambda}$ of
$x=(x_1,x_2,\dots,x_{n+1})_e$ is defined as follows
\begin{gather}\label{E+def}
D^+_{\lambda}(x):= {\det}^+ \left(
\begin{array}{cccc}
e^{2\pi i l_1x_1}&e^{2\pi i l_1x_2}&\dots&e^{2\pi i l_1x_{n+1}}\\
e^{2\pi i l_2x_1}&e^{2\pi i l_2x_2}&\dots&e^{2\pi i l_2x_{n+1}}\\[-1ex]
\vdots&\vdots&\ddots&\vdots\\[1 ex]
e^{2\pi i l_{n+1}x_1}&e^{2\pi i l_{n+1}x_2}&\dots&e^{2\pi i l_{n+1}x_{n+1}}
\end{array}
\right).
\end{gather}
Here, ${\det}^+$ is calculated as a conventional determinant,
except that all of its monomial terms are taken with positive sign.
It is also called \emph{permanent}~\cite{Henryk} or \emph{antideterminant}.
\end{definition}

It was shown in~\cite{KlimykPatera2007-3} that it suffices to
consider $D_{\lambda}^+(x)$ on the hyperplane $x\in {\mathcal H}$ (see~(\ref{plane H})).
Furthermore, due to the following property of the permanent
\begin{gather*}
{\det}^+ (a_{ij})_{i,j=1}^m=\sum_{s\in {\rm S}_m}a_{1,s(1)}a_{2,s(2)}\cdots a_{m,s(m)}
=\sum_{s\in {\rm S}_m}a_{s(1),1}a_{s(2),2}\cdots a_{s(m),m}
\end{gather*}
we have
\begin{gather*}
D^+_{\lambda}(x)= \sum_{s\in {\rm S}_{n+1}}e^{2\pi i l_1 x_{s(1)}} \cdots
e^{2\pi i l_m x_{s(n+1)}} =\sum_{s\in {\rm S}_{n+1}}e^{2\pi i
(\lambda,s(x))} =\sum_{s\in {\rm S}_{n+1}}e^{2\pi i (s(\lambda),x)}.
\end{gather*}

\begin{proposition}
For all $\lambda, x\in \mathcal{H}\subset\R^{n+1}$, we have the following connection between
the symmetric multivariate exponential functions in $n+1$ variables, and $C$ orbit functions of $A_n$
$D^+_{\lambda}(x)=kC_\lambda(x)$,
where $k=\tfrac{|W|}{|W_\lambda|}$, $|W|$ and $|W_\lambda|$ are sizes of the Weyl group and Weyl orbit respectively.
In particular, for generic points, $k=1$.
\end{proposition}
\begin{proof}
Proof follows from the definitions of the functions $C$ and $D^+$ (definitions 1 and 4 respectively)
and properties of orbit functions formulated in Section~\ref{orb_func_n+1}.
\end{proof}

\begin{definition}\cite{KlimykPatera2007-3}
For a fixed point $\lambda=(l_1,l_2,\dots,l_{n+1})_e$, such that
$l_1\geq l_2\geq\cdots\geq l_{n+1}$,
${\sum\limits_{k=1}^{n+1}l_k=0}$, the antisymmetric multivariate
exponential function $D^-_{\lambda}$ of
$x=(x_1,x_2,\dots,x_{n+1})_e\in \mathcal{H}$ is defined as follows
\begin{gather}\label{E+def}
D^-_{\lambda}(x):=\det \left(
\begin{array}{cccc}
e^{2\pi i l_1x_1}&e^{2\pi i l_1x_2}&\dots&e^{2\pi i l_1x_{n+1}}\\
e^{2\pi i l_2x_1}&e^{2\pi i l_2x_2}&\dots&e^{2\pi i l_2x_{n+1}}\\[-1ex]
\vdots&\vdots&\ddots&\vdots\\[1 ex]
e^{2\pi i l_{n+1}x_1}&e^{2\pi i l_{n+1}x_2}&\dots&e^{2\pi i l_{n+1}x_{n+1}}
\end{array}
\right)=\sum_{s\in {\rm S}_{n+1}}({\rm sgn}\ s)e^{2\pi i(s(\lambda),x)},
\end{gather}
where ${\rm sgn}$ is the permutation sign.
\end{definition}

\begin{proposition}
For all generic points $\lambda\in \mathcal{H}\subset\R^{n+1}$, we have the
following connection ${D^-_{\lambda}(x)=S_\lambda(x)}$.

The antisymmetric multivariate exponential functions $D^-$, and $S$
orbit functions, equal zero for non-generic points.
\end{proposition}
\begin{proof}
Proof directly follows from the definitions of functions $S$ and $D^-$
(definitions 2 and 5 respectively), and properties of $S$ functions
formulated in Section~\ref{orb_func_n+1}.
\end{proof}

\begin{definition}~\cite{KlimykPateraAlternatingExp}
The alternating multivariate exponential function
$D^{{\rm Alt}}_\lambda(x)$, for\linebreak
 $x=(x_1,\dots,x_{n+1})_e$, $\lambda=(l_1,\dots,l_{n+1})_e$,
is defined as the function
\begin{gather}\label{sdetD_Alt}
D^{{\rm Alt}}_\lambda(x):={\rm sdet}\left(
\begin{array}{cccc}
e^{2\pi i l_1x_1}    & e^{2\pi i l_1x_2}    & \cdots  &  e^{2\pi i l_1x_{n+1}}\\
e^{2\pi i l_2x_1}    & e^{2\pi i l_2x_2}    & \cdots  &  e^{2\pi i l_2x_{n+1}}\\[-1 ex]
\vdots                     & \vdots                     & \ddots  &  \vdots \\
e^{2\pi i l_{n+1}x_1}& e^{2\pi i l_{n+1}x_2}& \cdots  &  e^{2\pi i l_{n+1}x_{n+1}}
\end{array}
\right),
\end{gather}
where ${\rm Alt}_{n+1}$ is the alternating group (even subgroup of ${\rm S}_{n+1}$) and
\begin{gather*}
{\rm sdet} \left(e^{2\pi i l_jx_k}\right)_{j,k=1}^{n+1}:=\!\!\!
\sum_{w\in {\rm Alt}_{n+1}}\!\!e^{2\pi il_1 x_{w(1)}}e^{2\pi il_2 x_{w(2)}}\cdots e^{2\pi il_{n+1} x_{w(n+1)}}=\!\!\!
\sum_{w\in {\rm Alt}_{n+1}}\!\!e^{2\pi i(\lambda,w(x))}.
\end{gather*}
Here, $(\lambda,x)$ denotes the scalar product in the $(n+1)$-dimensional Euclidean space.
\end{definition}
Note that ${\rm Alt}_{m}$ consists of even substitutions of ${\rm S}_m$, and is usually denoted as $A_{m}$;
here we change the notation in order to avoid confusion with simple Lie algebra $A_n$ notations.

It was shown in~\cite{KlimykPateraAlternatingExp} that it is sufficient to consider the function
$D^{{\rm Alt}}_\lambda(x)$ on the hyperplane~$\mathcal{H}\colon x_1+x_2+\cdots +x_{n+1}=0$ for $\lambda$, such
that $l_1\ge l_2\ge l_3\ge \cdots \ge l_{n+1}$.

Alternating multivariate exponential functions are obviously connected with
symmetric and antisymmetric multivariate exponential functions.
This connection is the same as that of the cosine and sine, with the exponential function of
one variable ${D^{{\rm Alt}}_\lambda(x)=\tfrac 12 (D^{+}_\lambda(x)+D^{-}_\lambda(x))}$.

\begin{proposition}
For all generic points $\lambda\in \mathcal{H}\subset\R^{n+1}$,
the following relation between the alternative multivariate exponential functions~$D^{{\rm Alt}}$
and $E$-orbit functions of $A_n$ holds true: ${D^{{\rm Alt}}_\lambda(x)=E_\lambda(x)}$.

For non-generic points $\lambda$, we have $E_\lambda(x)=C_\lambda(x)$ and, therefore, $E_\lambda(x)=kD^+_{\lambda}(x)$,
where $k=\tfrac{|W|}{|W_\lambda|}$.
\end{proposition}
\begin{proof}
Proof directly follows from definitions 3 and 6, from the relation $E=\tfrac 12(C+S)$, and from the properties of orbit functions
formulated in Section~\ref{orb_func_n+1}.
\end{proof}

\section{Concluding remarks}
\begin{enumerate}
\item
Consequences of the identification of W-invariant orbit functions of compact simple Lie groups
and multivariable Chebyshev polynomials merit further exploitation. It is conceivable that Lie
theory may become a backbone of a segment of the theory of orthogonal polynomials of many variables.

Some of the properties of orbit functions translate readily into
properties of Chebyshev polynomials of many variables. However there
are other properties whose discovery from the theory of
polynomials is difficult to imagine. As an example, consider the
decomposition of the Chebyshev polynomial of the second kind into
the sum of Chebyshev polynomials of the first kind. In one variable,
it is a familiar problem that can be solved by elementary
means. For two and more variables, the problem turns out to be
equivalent to a more general question about representations of
simple Lie groups. In general the coefficients of that sum are the
dominant weight multiplocities. Again, simple specific cases can be
worked out, but a sophisticated algorithm is required to deal with
it in general \cite{MP1982}. In order to provide a solution for such a problem,
extensive tables have been prepared \cite{BMP} (see also references therein).
\smallskip

\item
Our approach to the derivation of multidimensional orthogonal
polynomials hinges on the knowledge of appropriate recursion
relations. The basic mathematical property underlying the existence
of the recursion relation is the complete decomposability of
products of the orbit functions. Numerous examples of the
decompositions of products of orbit functions, involving also other
Lie groups than ${\rm SU}(n)$, were shown elsewhere
\cite{KlimykPatera2006,KlimykPatera2007-1}. An equivalent problem is
the decomposition of products of Weyl group orbits \cite{HLP}.
\smallskip

\item
Possibility to discretize the polynomials is a consequence of the known
discretization of orbit functions. For orbit functions it is a
simpler problem, in that it is carried out in the real Euclidean space $\R^n$.
In principle, it carries over to the polynomials. But variables
of the polynomials happen to be on the maximal torus of the underlying
Lie group. Only in the case of $A_1$, the variables are real (the imaginary unit
multiplying the $S$-functions can be normalized away). For $A_n$ with
$n>1$ the functions are complex valued. Practical aspects of
discretization deserve to be thoroughly investigated.
\smallskip

\item
For simplicity of formulation, we insisted throughout this paper
that the underlying Lie group be simple. The extension to compact
semisimple Lie group and their Lie algebras is straightforward. Thus,
orbit functions are products of orbit functions of simple
constituents, and different types of orbit functions can be mixed.
\smallskip

\item
Polynomials formed from $E$-functions by the same substitution of
variables should be equally interesting once $n>1$. We know of no
analogs of such polynomials in the standard theory of polynomials
with more than one variable. Intuitively, they would be formed as
`halves' of Chebyshev polynomials although their domain of
orthogonality is twice as large as that of Chebyshev polynomials
\cite{KlimykPatera2008}.

\smallskip

\item
Orbit functions have many other properties
\cite{KlimykPatera2006,KlimykPatera2007-1,KlimykPatera2008} that can
now be rewritten as properties of Chebyshev polynomials. Let us
point out just that they are eigenfunctions of appropriate Laplace
operators with known eigenvalues.
\smallskip

\item
Notions of multivariate trigonometric functions \cite{KP2007} lead
us to the idea of new, yet to be defined classes of $W$-orbit
functions based on trigonometric sine and cosine functions, hence
also to new types of polynomials.
\smallskip

\item
Analogs of orbit functions of Weyl groups can be introduced also
for the finite Coxeter groups that are not Weyl groups of a
simple Lie algebra. Many of the properties of orbit functions extend
to these cases. Only their orthogonality, continuous or discrete,
has not been shown so far.
\smallskip

\item
Our choice of the $n$ dimensional subspace $\mathcal H$ in
$\R^{n+1}$ by requirement \eqref{plane H}, is not the only
possibility. A reasonable alternative appears to be setting
$l_{n+1}=0$ (orthogonal projection on $\R^n$).

\end{enumerate}

\subsection*{Acknowledgements}

The work was partially supported by the Natural Science and
Engineering Research Council of Canada. MITACS, and by the MIND
Research Institute, Calif. Two of us, M.N. and A.T., are grateful
for the hospitality extended to them at the Centre de Recherches
math\'ematiques, Universit\'e de Montr\'eal, where the work was
carried out.

The authors are grateful to A. Kiselev and the anonymous Journal referee
for critical remarks and comments on the previous version of
this paper.

\end{document}